\documentclass{amsproc}%
\usepackage{amsfonts}
\usepackage{amsmath}
\usepackage{amssymb}
\usepackage{graphicx}%
\providecommand{\U}[1]{\protect\rule{.1in}{.1in}}
\theoremstyle{plain}

\newtheorem{corollary}{Corollary}

\newtheorem{lemma}{Lemma}

\newtheorem{remark}{Remark}

\newtheorem{theorem}{Theorem}
\numberwithin{equation}{section}
\begin{document}
\title[Approximate Equivalence]{Approximate Equivalence in von Neumann Algebras}
\author{Qihui Li}
\address{East China University of Science and Technology, Shanghai, China}
\email{lqh991978@gmail.com}
\author{Don Hadwin}
\address{Mathematics Department, University of New Hampshire}
\email{operatorguy@gmail.com}
\author{Wenjing Liu}
\address{Mathematics Department, University of New Hampshire}
\email{wenjingtwins87@gmail.com}
\thanks{This author was partially supported by NSFC(Grant No.11671133).}
\thanks{This author was supported by a Collaboration Grant from the Simons Foundation}
\thanks{This author is supported by a grant from the Eric Nordgren Research Felloship Fund}
\thanks{This paper is in final form and no version of it will be submitted for
publication elsewhere.}
\subjclass{ }
\dedicatory{Dedicated to Lyra.}
\begin{abstract}
Suppose $\mathcal{A}$ is a separable unital ASH C*-algebra, $\mathcal{R}$ is a
sigma-finite II$_{\infty}$ factor von Neumann algebra, and $\pi,\rho
:\mathcal{A}\rightarrow\mathcal{R}$ are unital $\ast$-homomorphisms such that,
for every $a\in\mathcal{A}$, the range projections of $\pi\left(  a\right)  $
and $\rho\left(  a\right)  $ are Murray von Neuman equivalent in $\mathcal{R}%
$. We prove that $\pi$ and $\rho$ are approximately unitarily equivalent
modulo $\mathcal{K}_{\mathcal{R}}$, where $\mathcal{K}_{\mathcal{R}}$ is the
norm closed ideal generated by the finite projections in $\mathcal{R}$. We
also prove a very general result concerning approximate equivalence in
arbitrary finite von Neumann algebras.

\end{abstract}
\maketitle

\section{Introduction}

In 1977 D. Voiculescu \cite{V} proved a remarkable theorem concerning
approximate (unitary) equivalence of representations of a separable unital
C*-algebra on a separable Hilbert space. The beauty of the theorem is that the
characterization was in terms of purely algebraic terms. This was made
explicit in the reformulation of Voiculescu's theorem in \cite{H} in terms of rank.

\begin{theorem}
\label{rank}\cite{V}Suppose $B\left(  H\right)  $ is the set of operators on a
separable Hilbert space $H$ and $\mathcal{K}\left(  H\right)  $ is the ideal
of compact operators. Suppose $\mathcal{A}$ is a separable unital C*-algebra,
and $\pi,\rho:\mathcal{A}\rightarrow B\left(  H\right)  $ are unital $\ast
$-homomorphisms. The following are equivalent:

\begin{enumerate}
\item There is a sequence $\left\{  U_{n}\right\}  $ of unitary operators in
$B\left(  H\right)  $ such that

\begin{enumerate}
\item $U_{n}\pi\left(  a\right)  U_{n}^{\ast}-\rho\left(  a\right)
\in\mathcal{K}\left(  H\right)  $ for every $n\in\mathbb{N}$ and every
$a\in\mathcal{A}$.

\item $\left\Vert U_{n}\pi\left(  a\right)  U_{n}^{\ast}-\rho\left(  a\right)
\right\Vert \rightarrow0$ for every $a\in\mathcal{A}$.
\end{enumerate}

\item There is a sequence $\left\{  U_{n}\right\}  $ of unitary operators in
$B\left(  H\right)  $ such that, for every $a\in\mathcal{A}$,%
\[
\left\Vert U_{n}\pi\left(  a\right)  U_{n}^{\ast}-\rho\left(  a\right)
\right\Vert \rightarrow0.
\]

\item For every $a\in\mathcal{A}$,
\[
\text{\textrm{rank}}\left(  \pi\left(  a\right)  \right)  =\text{\textrm{rank}%
}\left(  \rho\left(  a\right)  \right)  \text{ .}%
\]

\end{enumerate}
\end{theorem}

If $\pi:\mathcal{A}\rightarrow B\left(  H\right)  $ is a unital $\ast
$-homomorphism, we will write $\pi\sim_{a}\rho$ in $B\left(  H\right)  $ to
mean that statement $\left(  2\right)  $ in the preceding theorem holds and we
will write $\pi\sim_{a}\rho$ ( $\mathcal{K}\left(  H\right)  $) in $B\left(
H\right)  $ to indicate statements $\left(  1\right)  $ and $\left(  2\right)
$ hold. When the C*-algebra $\mathcal{A}$ is not separable, $\pi\sim_{a}\rho$
means that there is a \emph{net} of unitaries $\left\{  U_{\lambda}\right\}  $
such that, for every $a\in\mathcal{A}$, $\left\Vert U_{\lambda}\pi\left(
a\right)  U_{\lambda}^{\ast}-\rho\left(  a\right)  \right\Vert \rightarrow0.$
It was shown in \cite{H} that $\pi\sim_{a}\rho$ if and only if \textrm{rank}%
$\left(  \pi\left(  a\right)  \right)  =$ \textrm{rank}$\left(  \rho\left(
a\right)  \right)  $ always holds even when $\mathcal{A}$ or $H$ is not
separable, where, for $T\in B\left(  H\right)  $, \textrm{rank}$\left(
T\right)  $ is the Hilbert-space dimension of the projection $\mathfrak{R}%
\left(  T\right)  $ onto the closure of the range of $T$.

Later Huiru Ding and the first author \cite{DH} extended the notion of rank to
operators in a von Neumann algebra $\mathcal{M}$, i.e., if $T\in\mathcal{M}$,
then $\mathcal{M}$-\textrm{rank}$\left(  T\right)  $ is the Murray von Neumann
equivalence class of the projection $\mathcal{\mathfrak{R}}\left(  T\right)
$. If $p$ and $q$ are projections in a C*-algebra $\mathcal{W}$, we write
$p\sim q$ in $\mathcal{W}$ to mean that there is a partial isometry
$v\in\mathcal{W}$ such that $v^{\ast}v=p$ and $vv^{\ast}=q$. Thus
$\mathcal{M}$-\textrm{rank}$\left(  T\right)  =\mathcal{M}$-\textrm{rank}%
$\left(  S\right)  $ if and only if $\mathfrak{R}\left(  S\right)
\sim\mathfrak{R}\left(  T\right)  $. In \cite{DH} they extended Voiculescu's
theorem for representations of a separable AH C*-algebra into a von Neumann
algebra on a separable Hilbert space, i.e., $\pi\sim\rho$ in $\mathcal{M}$ if
and only if, for every $a$,%
\[
\mathcal{M}\text{-}\mathrm{rank}\left(  \pi\left(  a\right)  \right)
=\mathcal{M}\text{-}\mathrm{rank}\left(  \rho\left(  a\right)  \right)
\text{.}%
\]
When the algebra $\mathcal{A}$ is ASH, their characterization works when the
von Neumann algebra is a $II_{1}$ factor \cite{DH} (See Theorem \ref{SAFE}.)

When $\mathcal{M}$ is a $\sigma$-finite type $II_{\infty}$ factor with a
faithful normal tracial weight $\tau$, there are analogues $\mathcal{F}%
_{\mathcal{M}}$ and $\mathcal{K}_{\mathcal{M}}$ of the finite rank operators
and compact operators; namely, $\mathcal{F}_{\mathcal{M}}$ is the ideal
generated by the projections $P\in\mathcal{M}$ with $\tau\left(  P\right)
<\infty$, and $\mathcal{K}_{\mathcal{M}}$ is the norm closure of
$\mathcal{F}_{\mathcal{M}}$. It happens that $\mathcal{F}_{\mathcal{M}}$ is
contained in every nonzero ideal of $\mathcal{M}$ and $\mathcal{K}%
_{\mathcal{M}}$ is the only nontrivial norm closed ideal of $\mathcal{M}$.
$\pi,\rho:\mathcal{A}\rightarrow\mathcal{M}$ are unital $\ast$-homomorphisms,
then $\pi\sim_{a}\rho$ ($\mathcal{K}_{\mathcal{M}}$) if and only if the
sequence $\left\{  U_{n}\right\}  $ of unitary operators in $\mathcal{M}$ for
which part $\left(  2\right)  $ in Theorem Qihui Li, Junhao Shen, and Rui Shi
holds can be chosen so that, for every $n\in\mathbb{N}$ and every
$a\in\mathcal{A}$,%
\[
U_{n}\pi\left(  a\right)  U_{n}^{\ast}-\rho\left(  a\right)  \in
\mathcal{K}_{\mathcal{M}}\text{ .}%
\]
This naturally leads to the following question:

\textbf{Question.} Suppose $\mathcal{M}$ is a $\sigma$-finite $II_{1}$ factor
von Neumann algebra, $\mathcal{A}$ is a separable unital C*-algebra, and
$\pi,\rho:\mathcal{A}\rightarrow\mathcal{M}$ are unital $\ast$-homomorphisms.
If $\pi\sim_{a}\rho$ in $\mathcal{M}$, does it follow that $\pi\sim_{a}\rho$
($\mathcal{K}_{\mathcal{M}}$)?

In \cite{HS} Rui Shi and the first author gave an affirmative answer to this
question when $\mathcal{A}$ is commutative. More recently this result has been
extended by Shilin Wen, Junsheng Fang and Rui Shi \cite{FSW} to the case when
$\mathcal{A}$ is an AF C*-algebra, and in \cite{SR} to the case when
$\mathcal{A}$ is AH. In this paper we extend the results to the case when
$\mathcal{A}$ is a ASH C*-algebra (Theorem \ref{ASHmain}) . We also extend the
results in some of the results in \cite{DH} and \cite{FSW} for arbitrary
finite von Neumann algebras and for the new class of \emph{strongly
LF-embeddable} C*-algebras (Theorem \ref{SAFE} and Theorem \ref{summand}).

The proof of Voiculescu's theorem (Theorem \ref{rank}) has two parts. The
'hard" part is showing that if $\mathcal{A}\subset B\left(  \ell^{2}\right)  $
is a separable unital C*-algebra, and $\pi:\mathcal{A}\rightarrow B\left(
\ell^{2}\right)  $ is a unital $\ast$-homomorphism such that $\mathcal{K}%
\left(  \ell^{2}\right)  \subset\ker\pi,$ then%
\[
id_{\mathcal{A}}\oplus\pi\sim_{a}id_{\mathcal{A}}\text{ (}\mathcal{K}\left(
\ell^{2}\right)  \text{).}%
\]
The "easy part" involves the compact operators. Suppose $\mathcal{A}$ is a
separable unital C*-algebra and $\pi:\mathcal{A\rightarrow B}\left(  \ell
^{2}\right)  $ is a unital $\ast$-homomorphism. Then $\sup\left\{
\mathcal{\mathfrak{R}}\left(  \pi\left(  a\right)  \right)  :\pi\left(
a\right)  \in\mathcal{K}\left(  \ell^{2}\right)  \right\}  $ reduces $\pi$ and
leads to a decomposition
\[
\pi=\pi_{0}\oplus\pi_{1}.
\]
The "easy part" says that if $\pi\sim_{a}\rho$, then $\pi_{0}$ and $\rho_{0}$
must be unitarily equivalent. In \cite{LSS} Qihui Li, Junhao Shen, and Rui Shi
prove an analogue of the "hard part" when $\mathcal{A}$ is nuclear for a
$II_{\infty}$ factor. The "easy part" for von Neumann algebras is much more
difficult for $II_{\infty}$ factors because, there is a complete
characterization of C*-subalgebras of $\mathcal{K}\left(  \ell^{2}\right)  $
and their representations which is totally lacking for C*-subalgebras of
$\mathcal{K}_{\mathcal{M}}$. In fact, the "easy part" isn't true for a
$II_{\infty}$ factor. The analogue must be
\[
\pi_{0}\sim_{a}\rho_{0}\text{ (}\mathcal{K}_{\mathcal{M}}\text{).}%
\]
We prove a general analogue of the "easy part" for $II_{\infty}$ factors when
the C*-algebra is ASH (Theorem \ref{main}).

\section{Finite von Neumann Algebras}

A separable C*-algebra is AF if it is a direct limit of finite-dimensional
C*-algebras. A separable C*-algebra is \emph{homogeneous} if it is a finite
direct sum of algebras of the form $\mathbb{M}_{n}\left(  C\left(  X\right)
\right)  ,$ where $X$ is a compact metric space. A unital C*-algebra is
\emph{subhomogeneous} if there is an $n\in\mathbb{N}$, such that every
representation is on a Hilbert space of dimension at most $n$; equivalently,
if $x^{n}=0$ for every nilpotent $x\in\mathcal{A}$. Every subhomogeneous
algebra is a subalgebra of a homogeneous one. Every subhomogeneous von Neumann
algebra is homogeneous; in particular, if $\mathcal{A}$ is subhomogeneous,
then $\mathcal{A}^{\# \#}$ is homogeneous. A C*-algebra is approximately
subhomogeneous (ASH) if it is a direct limit of subhomogeneous C*algebras.

There has been a lot of work determining which separable C*-algebras are
AF-embeddable. A (possibly nonseparable) C*-algebra $\mathcal{B}$ is \emph{LF}
if, for every finite subset $F\subset\mathcal{B}$ and every $\varepsilon>0$
there is a finite-dimensional C*-algebra $\mathcal{D}$ of $\mathcal{B}$ such
that, for every $b\in F$, \textrm{dist}$\left(  b,\mathcal{D}\right)
<\varepsilon$. Every separable unital C*-subalgebra of a LF C*-algebra is
contained in a separable AF subalgebra.

We are interested in a more general property. We say that a unital C*-algebra
$\mathcal{A}$ is \emph{strongly LF-embeddable} if there is an LF C*-algebra
$\mathcal{B}$ such that $\mathcal{A}\subset\mathcal{B}\subset\mathcal{A}^{\#
\#}$. It is easily shown that an ASH algebra is strongly LF-embeddable, i.e.,
if $\left\{  \mathcal{A}_{\lambda}\right\}  $ is an increasingly directed
family of subhomogeneous C*-algebas and $\mathcal{A}=\left(  \cup_{\lambda
}\mathcal{A}_{\lambda}\right)  ^{-\left\Vert {}\right\Vert }$, then
$\mathcal{A}\subset\left(  \cup_{\lambda}\mathcal{A}_{\lambda}^{\# \#}\right)
^{-\left\Vert {}\right\Vert }\subset\mathcal{A}^{\# \#}$.

\begin{lemma}
\label{LF}Suppose $\mathcal{B}$ is a unital LF C*-algebra and $\mathcal{D}$
$=$ $\mathbb{M}_{n_{1}}\left(  \mathbb{C}\right)  \oplus\cdots\oplus
\mathbb{M}_{nk}\left(  \mathbb{C}\right)  $ and $\mathcal{W}$ is a unital C*-algebra.

\begin{enumerate}
\item If $\pi,\rho:\mathcal{D}\rightarrow\mathcal{W}$ are unital $\ast
$-homomorphisms and $\pi\left(  e_{11,s}\right)  \sim\rho\left(
e_{11,s}\right)  $ for $1\leq s\leq k,$ where $\left\{  e_{ij,s}\right\}  $ is
the system of matrix units for $\mathbb{M}_{ns}\left(  \mathbb{C}\right)  $,
then $\pi$ and $\rho$ are unitarily equivalent in $\mathcal{W}$.

\item If $\pi,\rho:\mathcal{B}\rightarrow\mathcal{W}$ are unital $\ast
$-homomorphisms such that $\pi\left(  p\right)  \sim\rho\left(  p\right)  $ in
$\mathcal{W}$ for every projection $p\in\mathcal{B}$, then $\pi\sim_{a}\rho$
in $\mathcal{W}$.
\end{enumerate}
\end{lemma}

\begin{proof}
$\left(  1\right)  $ Since $e_{ii,s}\sim e_{11,s}$ in $\mathcal{D}$ for $1\leq
i\leq n_{s}$ and $1\leq s\leq k$, we see that $\pi\left(  e_{ii,s}\right)
\sim\rho\left(  e_{ii,s}\right)  $ in $\mathcal{W}$ for $1\leq i\leq n_{s}$
and $1\leq s\leq k.$ It follows from \cite[Theorem 2]{DH} that $\pi$ and
$\rho$ are unitarily equivalent in $\mathcal{W}$.

$\left(  2\right)  $ Suppose $\Lambda$ is the set of all pairs $\lambda
=\left(  F_{\lambda},\varepsilon_{\lambda}\right)  $ with $F_{\lambda}$ a
finite subset of $\mathcal{B}$ and $\varepsilon_{\lambda}>0$. Clearly
$\Lambda$ is directed by $\left(  \subset,\geq\right)  $. For $\lambda
\in\Lambda$, we can choose a finite-dimensional algebra $\mathcal{D}_{\lambda
}\subset\mathcal{B}$ such that, for every $x\in F_{\lambda},$ \textrm{dist}%
$\left(  x,\mathcal{D}_{\lambda}\right)  <\varepsilon_{\lambda}$. It follows
from part $\left(  1\right)  $ that there is a unitary operator $U_{\lambda
}\in\mathcal{W}$ such that, for every $x\in F_{\lambda}$, $U\pi\left(
x\right)  U^{\ast}=\rho\left(  x\right)  $. For each $a\in F_{\lambda,}we$ can
choose $x_{a}\in\mathcal{D}_{\lambda}$ such that $\left\Vert a-x_{a}%
\right\Vert <\varepsilon_{\lambda..}$ Hence, for every $a\in F_{\lambda}$%
\[
\left\Vert U_{\lambda}\pi\left(  a\right)  U_{\lambda}^{\ast}-\rho\left(
a\right)  \right\Vert =\left\Vert U_{\lambda}\pi\left(  a-x_{\lambda}\right)
U_{\lambda}^{\ast}-\rho\left(  a-x_{\lambda}\right)  \right\Vert
<2\varepsilon_{\lambda}.
\]
It follows that, for every $a\in\mathcal{A}$,%
\[
\lim_{\lambda}\left\Vert U_{\lambda}\pi\left(  a\right)  U^{\ast}-\rho\left(
a\right)  \right\Vert =0.
\]

\end{proof}

The key property of a finite von Neumann algebra $\mathcal{M}$ is that there
is a faithful normal tracial conditional expectation $\Phi$ from $\mathcal{M}$
to its center $\mathcal{Z}\left(  \mathcal{M}\right)  $, and that for
projections $p$ and $q$ in $\mathcal{M}$, we have $p$ and $q$ are Murray-von
Neumann equivalent if and only if $\Phi\left(  p\right)  =\Phi\left(
q\right)  $. Note that in the next lemma and the theorem that follows, there
is no separability assumption on the C*-algebra $\mathcal{A}$ or the dimension
of the Hilbert space on which $\mathcal{M}$ acts.

\begin{lemma}
\label{trace}Suppose $\mathcal{A}$ is a (possibly nonunital) C*-algebra,
$\mathcal{M}$ is a finite von Neumann algebra with center-valued trace
$\Phi:\mathcal{M}\rightarrow\mathcal{Z}\left(  \mathcal{M}\right)  $. If
$\pi,\rho:\mathcal{A}\rightarrow\mathcal{M}$ are $\ast$-homomorphisms such
that, for every $a\in\mathcal{A}$,%
\[
\mathcal{M}\text{-rank}\left(  \pi\left(  a\right)  \right)  =\mathcal{M}%
\text{-rank}\left(  \rho\left(  a\right)  \right)  ,
\]
then
\[
\Phi\circ\pi=\Phi\circ\rho\text{ .}%
\]

\end{lemma}

\begin{proof}
We can extend $\pi$ and $\rho$ to weak*-weak* continuous *-homomorphisms
$\hat{\pi},\hat{\rho}:\mathcal{A}^{\# \#}\rightarrow\mathcal{M}$. Suppose
$x\in\mathcal{A}$ and $0\leq x\leq1$. Suppose $0<\alpha<1$ and define
$f_{\alpha}:\left[  0,1\right]  \rightarrow\left[  0,1\right]  $ by
\[
f\left(  t\right)  =dist\left(  t,\left[  0,\alpha\right]  \right)  .
\]
Since $f\left(  0\right)  =0,$ we see that $f\left(  x\right)  \in\mathcal{A}%
$, and $\chi_{(\alpha,1]}\left(  x\right)  =$ weak*-$\lim_{n\rightarrow\infty
}f\left(  x\right)  ^{1/n}\in\mathcal{A}^{\# \#}$, so
\[
\mathfrak{R}\left(  f\left(  x\right)  \right)  =\chi_{(\alpha,1]}\left(
x\right)  \text{ .}%
\]
It follows that%
\[
\hat{\pi}\left(  \chi_{(\alpha,1]}\left(  x\right)  \right)  =\mathfrak{R}%
\left(  \pi\left(  f_{\alpha}\left(  x\right)  \right)  \right)
=\chi_{(\alpha,1]}\left(  \pi\left(  x\right)  \right)  \text{ }%
\]
and%
\[
\hat{\rho}\left(  \chi_{(\alpha,1]}\left(  x\right)  \right)  =\mathfrak{R}%
\left(  \rho\left(  f_{\alpha}\left(  x\right)  \right)  \right)
=\chi_{(\alpha,1]}\left(  \rho\left(  x\right)  \right)  \text{.}%
\]
Hence
\[
\Phi\left(  \hat{\pi}\left(  \chi_{(\alpha,1]}\left(  x\right)  \right)
\right)  =\Phi\left(  \hat{\rho}\left(  \chi_{(\alpha,1]}\left(  x\right)
\right)  \right)  .
\]
Suppose $0<\alpha<\beta<1$. Since $\chi_{(\alpha,\beta]}=\chi_{(\alpha
,1]}-\chi_{(\beta,1]},$ we see that%
\[
\Phi\left(  \hat{\pi}\left(  \chi_{(\alpha,\beta]}\left(  x\right)  \right)
\right)  =\Phi\left(  \hat{\rho}\left(  \chi_{(\alpha,\beta]}\left(  x\right)
\right)  \right)  .
\]
Thus, for all $n\in\mathbb{N}$,%
\[
\Phi\left(  \hat{\pi}\left(  \sum_{k-1}^{n-1}\frac{k}{n}\chi_{(\frac{k}%
{n},\frac{k+1}{n}]}\left(  x\right)  \right)  \right)  =\Phi\left(  \hat{\rho
}\left(  \sum_{k-1}^{n-1}\frac{k}{n}\chi_{(\frac{k}{n},\frac{k+1}{n}]}\left(
x\right)  \right)  \right)  .
\]
Since, for every $n\in\mathbb{N}$,
\[
\left\Vert x-\sum_{k-1}^{n-1}\frac{k}{n}\chi_{(\frac{k}{n},\frac{k+1}{n}%
]}\left(  x\right)  \right\Vert \leq1/n\text{,}%
\]
it follows that
\[
\Phi\left(  \pi\left(  x\right)  \right)  =\Phi\left(  \hat{\pi}\left(
x\right)  \right)  =\Phi\left(  \hat{\rho}\left(  x\right)  \right)
=\Phi\left(  \rho\left(  x\right)  \right)  .
\]
Since $\mathcal{A}$ is the linear span of its positive contractions,
$\Phi\circ\pi=\Phi\circ\rho$.
\end{proof}

\begin{theorem}
\label{SAFE}Suppose $\mathcal{A}$ is strongly LF-embeddable, $\mathcal{M}$ is
a finite von Neumann algebra with center-valued trace $\Phi:\mathcal{M}%
\rightarrow\mathcal{Z}\left(  \mathcal{M}\right)  $. If $\pi,\rho
:\mathcal{A}\rightarrow\mathcal{M}$ are unital $\ast$-homomorphisms, then the
following are equivalent:

\begin{enumerate}
\item $\pi\sim_{a}\rho$ $\left(  \mathcal{M}\right)  .$

\item $\mathcal{M}$-rank$\left(  \pi\left(  a\right)  \right)  =\mathcal{M}%
$-rank$\left(  \rho\left(  a\right)  \right)  $ for every $a\in\mathcal{A}$.

\item $\Phi\circ\pi=\Phi\circ\rho.$
\end{enumerate}
\end{theorem}

\begin{proof}
$\left(  3\right)  \Rightarrow\left(  1\right)  $. We can extend $\pi$ and
$\rho$ to weak*-weak* continuous $\ast$-homomorphisms $\hat{\pi},\hat{\rho
}:\mathcal{A}^{\# \#}\rightarrow\mathcal{M}$. Since $\Phi$ is weak*-weak*
continuous, it follows that $\Phi\circ\hat{\pi}=\Phi\circ\hat{\rho}$. Since
$\mathcal{A}$ is strongly LF-embeddable, there is an LF algebra $\mathcal{B}$
such that $\mathcal{A}\subset\mathcal{B}\subset\mathcal{A}^{\# \#}$. For every
projection $p\in\mathcal{B}$ we have%
\[
\Phi\left(  \hat{\pi}\left(  p\right)  \right)  =\Phi\left(  \hat{\rho}\left(
p\right)  \right)  ,
\]
which implies that $\hat{\pi}\left(  p\right)  \sim\hat{\rho}\left(  p\right)
.$ Hence, by Lemma \ref{LF}, $\hat{\pi}|_{\mathcal{B}}\sim_{a}\hat{\rho
}|_{\mathcal{B}}$ in $\mathcal{M}$. Thus $\pi\sim_{a}\rho$ $\left(
\mathcal{M}\right)  $.

$\left(  1\right)  \Rightarrow\left(  3\right)  .$ Suppose $\left\{
U_{\lambda}\right\}  $ is a net of unitaries in $\mathcal{M}$ such that, for
every $a\in\mathcal{A}$,
\[
\left\Vert U_{\lambda}\pi\left(  a\right)  U_{\lambda}^{\ast}-\rho\left(
a\right)  \right\Vert \rightarrow0.
\]
Thus, since $\Phi$ is tracial and continuous,
\[
\Phi\left(  \rho\left(  a\right)  \right)  =\lim_{\lambda}\Phi\left(
U_{\lambda}\pi\left(  a\right)  U_{\lambda}^{\ast}\right)  =\Phi\left(
\pi\left(  a\right)  \right)  .
\]

$\left(  3\right)  \Rightarrow\left(  2\right)  .$ Assume $\left(  3\right)
.$ Then, for any $a\in\mathcal{A}$,%
\[
\Phi\left(  \mathfrak{R}\left(  \pi\left(  a\right)  \right)  \right)
=\lim_{n\rightarrow\infty}\Phi\left(  \pi\left(  \left(  aa^{\ast}\right)
^{1/n}\right)  \right)  =\lim_{n\rightarrow\infty}\Phi\left(  \rho\left(
\left(  aa^{\ast}\right)  ^{1/n}\right)  \right)  =\Phi\left(  \mathfrak{R}%
\left(  \pi\left(  a\right)  \right)  \right)  .
\]
Hence $\mathfrak{R}\left(  \pi\left(  a\right)  \right)  \sim\mathfrak{R}%
\left(  \rho\left(  a\right)  \right)  .$ Thus $\mathcal{M}$-rank$\left(
\pi\left(  a\right)  \right)  =\mathcal{M}$-rank$\left(  \rho\left(  a\right)
\right)  $.

$\left(  2\right)  \Rightarrow\left(  3\right)  .$ This is Lemma \ref{trace}.
\end{proof}

\bigskip

\begin{remark}
\label{nonunital}It is important to note that the proof of $\left(  2\right)
\Rightarrow\left(  3\right)  $ holds even when $\mathcal{A}$ is not unital.
\end{remark}

In \cite{H} it was shown that if $\mathcal{A}$ is a separable unital
C*-algebra and $\pi$ and $\rho$ are representations on a separable Hilbert
space such that, for every $x\in\mathcal{A}$
\[
\text{\textrm{rank}}\pi\left(  x\right)  \leq\text{\textrm{rank}}\rho\left(
x\right)  ,
\]
then there is a representation $\sigma$ such that%
\[
\pi\oplus\sigma\sim_{a}\rho.
\]
In \cite{HS}, Rui Shi and the first author proved an analogue for
representations of separable abelian C*-algebras into $II_{1}$ factor von
Neumann algebras. This result was extended by Shilin Wen, Junsheng Fang and
Rui Shi \cite{FSW} to separable AF C*-algebras. We extend this result further,
including separable ASH C*-algebras.\bigskip

\begin{theorem}
\label{summand}Suppose $\mathcal{A}$ is a separable strongly LF-embeddable
C*-algebra and $\mathcal{M}$ is a $II_{1}$ factor von Neumann algebra with a
faithful normal tracial state $\tau$. Suppose $P$ is a projection in
$\mathcal{M}$ and $\pi:\mathcal{A}\rightarrow P\mathcal{M}P$ and
$\rho:\mathcal{A}\rightarrow\mathcal{M}$ are unital $\ast$-homomorphisms such
that, for every $a\in\mathcal{A}$,%
\[
\mathcal{M}\text{-rank}\left(  \pi\left(  a\right)  \right)  \leq
\mathcal{M}\text{-rank}\left(  \rho\left(  a\right)  \right)  .
\]

Then there is a unital $\ast$-homomorphism $\sigma:\mathcal{A}\rightarrow
P^{\perp}\mathcal{M}P^{\perp}$ such that
\[
\pi\oplus\sigma\sim_{a}\rho\text{ }\left(  \mathcal{M}\right)  \text{ .}%
\]

\end{theorem}

\begin{proof}
As in the proof of Theorem \ref{SAFE} choose a separable AF C*-algebra
$\mathcal{B}$ such that $\mathcal{A}\subset\mathcal{B}\subset\mathcal{A}^{\#
\#}$, and extend $\pi$ and $\rho$ to unital weak*-weak* continuous $\ast
$-homomorphisms $\hat{\pi}$ and $\hat{\rho}$ with domain $\mathcal{A}^{\# \#}%
$. It was shown in \cite{DH} that the condition on $\pi$ and $\rho$ is
equivalent to: for every $a\in\mathcal{M}$ with $0\leq a$, $\tau\left(
\pi\left(  a\right)  \right)  \leq\tau\left(  \rho\left(  a\right)  \right)
$. It follows from weak* continuity that, for every $a\in\mathcal{A}^{\# \#}$
with $0\leq a$, $\tau\left(  \hat{\pi}\left(  a\right)  \right)  \leq
\tau\left(  \hat{\rho}\left(  a\right)  \right)  $. In particular this holds
for $0\leq a\in\mathcal{B}$. However, since $\mathcal{B}$ is AF, it follows
from \cite{FS} that there is a unital $\ast$-homomorphism $\gamma
:\mathcal{B}\rightarrow P^{\perp}\mathcal{A}P^{\perp}$ such that%
\[
\left(  \hat{\pi}|_{\mathcal{B}}\right)  \oplus\gamma\sim_{a}\hat{\rho
}|_{\mathcal{B}}\text{ }\left(  \mathcal{M}\right)  \text{.}%
\]
If we let $\sigma=\gamma|_{\mathcal{A}}$, we see $\pi\oplus\sigma\sim_{a}\rho$
$\left(  \mathcal{M}\right)  $.
\end{proof}

\section{Representations of ASH algebras relative to ideals}

We prove a version of Voiculescu's theorem for representations of a separable
ASH C*-algebras into sigma-finite type $II_{\infty}$ factor von Neumann
algebras. We first prove a more general result. We begin with a probably
well-known lemma.

\begin{lemma}
\label{ideal}Suppose $\mathcal{J}$ is a norm closed two-sided ideal in a von
Neumann algebra $\mathcal{M}$ and $\mathcal{J}_{0}$ is the ideal in
$\mathcal{M}$ generated by the projections in $\mathcal{J}$. Suppose also that
$\mathcal{A}$ is a C*-algebra and $\pi,\rho:\mathcal{A}\rightarrow\mathcal{M}$
are unital $\ast$-homomorphisms. Then

\begin{enumerate}
\item $\mathcal{J}$ is the norm closed linear span of the set of projections
in $\mathcal{J}$, so
\[
\mathcal{J}_{0}^{-\left\Vert {}\right\Vert }=\mathcal{J}\text{,}%
\]

\item $\mathcal{J}_{0}=\left\{  T\in\mathcal{M}:T=PTP\text{ for some
projection }P\in\mathcal{J}\right\}  $,

\item $T\in\mathcal{J}_{0}$ if and only if $\chi_{\left(  0,\infty\right)
}\left(  \left\vert T\right\vert \right)  =\mathfrak{R}\left(  T\right)
\in\mathcal{J}_{0}$,

\item If $P$ and $Q$ are projections in $\mathcal{J}_{0}$ then $P\vee
Q=\mathfrak{R}\left(  P+Q\right)  \in\mathcal{J}_{0}$,

\item $\pi^{-1}\left(  \mathcal{J}_{0}\right)  ^{-\left\Vert {}\right\Vert
}=\pi^{-1}\left(  \mathcal{J}\right)  $,

\item If $\left\{  \mathcal{A}_{i}:i\in I\right\}  $ is an increasingly
directed family of unital C*-subalgebras of $\mathcal{A}$ and $\mathcal{A}%
=\left[  \cup_{i\in I}\mathcal{A}_{i}\right]  ^{-\left\Vert {}\right\Vert },$
then%
\[
\left[  \cup_{i\in I}\mathcal{A}_{i}\cap\pi^{-1}\left(  \mathcal{J}%
_{0}\right)  \right]  ^{-\left\Vert {}\right\Vert }=\pi^{-1}\left(
\mathcal{J}\right)  \text{ .}%
\]

\end{enumerate}
\end{lemma}

\begin{proof}
$\left(  1\right)  ,\left(  2\right)  ,$ $\left(  3\right)  $ can be found in
\cite{KR2}.

$\left(  4\right)  .$ Suppose $a\in\pi^{-1}\left(  \mathcal{J}\right)  $. Then
$\pi\left(  a\right)  \in\mathcal{J}$, so
\[
\pi\left(  g_{\varepsilon}\left(  \left\vert a\right\vert \right)  \right)
=g_{\varepsilon}\left(  \left\vert \pi\left(  a\right)  \right\vert \right)
\chi_{\left(  \varepsilon/2,\infty\right)  }\left(  \left\vert \pi\left(
a\right)  \right\vert \right)  \in\mathcal{J}_{0},
\]
and
\[
\left\Vert a-ag_{\varepsilon}\left(  \left\vert a\right\vert \right)
\right\Vert \leq\varepsilon.
\]

$\left(  5\right)  .$ Let $\eta:\mathcal{M}\rightarrow\mathcal{M}/\mathcal{J}$
be the quotient map. Suppose $a\in\pi^{-1}\left(  \mathcal{J}\right)  $ and
$\varepsilon>0.$ Then there is an $i\in I$ and a $b\in\mathcal{A}_{i}$ such
that $\left\Vert a-b\right\Vert <\varepsilon$. Thus
\[
\left\Vert \left(  \eta\circ\left(  \pi|_{\mathcal{A}_{i}}\right)  \right)
\left(  b\right)  \right\Vert =\left\Vert \left(  \eta\circ\pi\right)  \left(
b\right)  \right\Vert =\left\Vert \left(  \eta\circ\pi\right)  \left(
b-a\right)  \right\Vert \leq\varepsilon,
\]
so there is a $w\in\mathcal{A}_{i}$ so that%
\[
\left\Vert w\right\Vert =\left\Vert \left(  \eta\circ\left(  \pi
|_{\mathcal{A}_{i}}\right)  \right)  \left(  w\right)  \right\Vert =\left\Vert
\left(  \eta\circ\left(  \pi|_{\mathcal{A}_{i}}\right)  \right)  \left(
b\right)  \right\Vert \leq\varepsilon.
\]
$z=b-w\in\ker\left(  \eta\circ\left(  \pi|_{\mathcal{A}_{i}}\right)  \right)
=\pi^{-1}\left(  \mathcal{J}\right)  \cap\mathcal{A}_{i}$, and $\left\Vert
b-z\right\Vert =\left\Vert w\right\Vert <\varepsilon$. It follows from part
$\left(  2\right)  $ that there is a $v\in\pi^{-1}\left(  \mathcal{J}%
_{0}\right)  \cap\mathcal{A}_{i}$ such that $\left\Vert z-v\right\Vert
\leq\varepsilon.$ Hence $\left\Vert a-v\right\Vert \leq\left\Vert
a-b\right\Vert +\left\Vert b-z\right\Vert +\left\Vert z-v\right\Vert
\leq3\varepsilon.$

$\left(  6\right)  .$ Let $\eta:\mathcal{M}\rightarrow\mathcal{M}/\mathcal{J}$
be the quotient map. Suppose $a\in\pi^{-1}\left(  \mathcal{J}\right)  $ and
$\varepsilon>0.$ Then there is an $i\in I$ and a $b\in\mathcal{A}_{i}$ such
that $\left\Vert a-b\right\Vert <\varepsilon$. Thus
\[
\left\Vert \left(  \eta\circ\left(  \pi|_{\mathcal{A}_{i}}\right)  \right)
\left(  b\right)  \right\Vert =\left\Vert \left(  \eta\circ\pi\right)  \left(
b\right)  \right\Vert =\left\Vert \left(  \eta\circ\pi\right)  \left(
b-a\right)  \right\Vert \leq\varepsilon,
\]
so there is a $w\in\mathcal{A}_{i}$ so that%
\[
\left\Vert w\right\Vert =\left\Vert \left(  \eta\circ\left(  \pi
|_{\mathcal{A}_{i}}\right)  \right)  \left(  w\right)  \right\Vert =\left\Vert
\left(  \eta\circ\left(  \pi|_{\mathcal{A}_{i}}\right)  \right)  \left(
b\right)  \right\Vert \leq\varepsilon.
\]
$z=b-w\in\ker\left(  \eta\circ\left(  \pi|_{\mathcal{A}_{i}}\right)  \right)
=\pi^{-1}\left(  \mathcal{J}\right)  \cap\mathcal{A}_{i}$, and $\left\Vert
b-z\right\Vert =\left\Vert w\right\Vert <\varepsilon$. It follows from part
$\left(  5\right)  $ that there is a $v\in\pi^{-1}\left(  \mathcal{J}%
_{0}\right)  \cap\mathcal{A}_{i}$ such that $\left\Vert z-v\right\Vert
\leq\varepsilon.$ Hence $\left\Vert a-v\right\Vert \leq\left\Vert
a-b\right\Vert +\left\Vert b-z\right\Vert +\left\Vert z-v\right\Vert
\leq3\varepsilon$.
\end{proof}

\bigskip

Suppose $\mathcal{A}$ is a unital C*-algebra, $\mathcal{M}\subset B\left(
H\right)  $ is a von Neumann algebra with a norm-closed ideal $\mathcal{J}$
and $\pi:\mathcal{A}\rightarrow\mathcal{M}$ is a unital $\ast$-homomorphism.
We define
\[
H_{\pi,\mathcal{J}}=\text{\textrm{sp}}^{-\left\Vert {}\right\Vert }\left(
\cup\left\{  \text{\textrm{ran}}\pi\left(  a\right)  :a\in\mathcal{A}\text{
and }\pi\left(  a\right)  \in\mathcal{J}\right\}  \right)  \text{.}%
\]
It is clear that $H_{\pi,\mathcal{J}}$ is a reducing subspace for $\pi$ and we
call the summand $\pi\left(  \cdot\right)  |_{H_{\pi,\mathcal{J}}}%
=\pi_{\mathcal{J}}$.

In Voiculescu's theorem, where $\pi,\rho:\mathcal{A}\rightarrow B\left(
H\right)  $ and $\mathcal{A}$ and $H$ are separable, we write%
\[
\pi=\pi_{\mathcal{K}\left(  H\right)  }\oplus\pi_{1},\text{ }\rho
=\rho_{\mathcal{K}\left(  H\right)  }\oplus\rho_{1}.
\]
The proof of Voiculescu's theorem involves showing
\[
\pi\sim_{a}\pi\oplus\rho_{1}=\pi_{\mathcal{K}\left(  H\right)  }\oplus\pi
_{1}\oplus\rho_{1},
\]
and%
\[
\rho\sim_{a}\rho\oplus\pi_{1}\backsimeq\rho_{\mathcal{K}\left(  H\right)
}\oplus\pi_{1}\oplus\rho_{1},
\]
which was the hard part. Using descriptions of C*-algebras of compact
operators and their representations, it is not too hard to show that the
equality of rank conditions imply that $\pi_{\mathcal{K}\left(  H\right)  }$
and $\rho_{\mathcal{K}\left(  H\right)  }$ are unitarily equivalent. When
$B\left(  H\right)  $ is replaced with a sigma-finite type $II_{\infty}$
factor von Neumann algebra $\mathcal{M}$ and $\mathcal{K}\left(  H\right)  $
is replaced with the closed ideal $\mathcal{K}_{\mathcal{M}}$ generated by the
finite projections, the hard part is harder (and unsolved) and the easy part
is not true.

In a deep and beautiful paper \cite{LSS} of Qihui Li, Junhao Shen, and Rui Shi
proved the best-to-date attack of the hard part.

\begin{theorem}
\label{nuclear}\cite{LSS} Suppose $\mathcal{A}$ is a separable nuclear
C*-algebra, $\mathcal{M}$ is a sigma-finite type $II_{\infty}$ factor von
Neumann algebra and $\pi,\sigma:\mathcal{A}\rightarrow\mathcal{M}$ are unital
$\ast$-homomorphisms such that
\[
\sigma|_{\pi^{-1}\left(  \mathcal{K}_{\mathcal{M}}\right)  }=0.
\]
Then
\[
\pi\sim_{a}\pi\oplus\sigma\text{ }\left(  \mathcal{K}_{\mathcal{M}}\right)
\text{ .}%
\]

\end{theorem}

\bigskip

The following is a fairly general version of the analogue of the "easy part"
of the proof of Voiculescu's theorem when the C*-algebra is ASH. In
particular, there is no assumption that the von Neumann algebra $\mathcal{M}$
is a sigma-finite or acts on a separable Hilbert space.

\begin{theorem}
\label{main}Suppose $\mathcal{A}$ is a separable unital ASH C*-algebra,
$\mathcal{M}\subset B\left(  H\right)  $ is a von Neumann algebra with a norm
closed two-sided ideal $\mathcal{J}$. Suppose $\pi,\rho:\mathcal{A}%
\rightarrow\mathcal{M}$ are unital $\ast$-homomorphisms such that

\begin{enumerate}
\item Every projection in $\mathcal{J}$ is finite,

\item $\mathcal{M}$-rank$\left(  \pi\left(  a\right)  \right)  =\mathcal{M}%
$-rank$\left(  \rho\left(  a\right)  \right)  $ for every $a\in\mathcal{A}$.
\end{enumerate}

Then there is a sequence $\left\{  W_{n}\right\}  $ of partial isometries in
$\mathcal{M}$ such that

\begin{enumerate}
\item[3.] $W_{n}^{\ast}W_{n}$ is the projection onto $H_{\pi,\mathcal{J}}$ and
$W_{n}W_{n}^{\ast}$ is the projection onto $H_{\rho,\mathcal{J},},$

\item[4.] $W_{n}\pi_{\mathcal{J}}\left(  a\right)  W_{n}^{\ast}-\rho
_{\mathcal{J}}\left(  a\right)  \in\mathcal{J}$ for every $n\in\mathbb{N}$ and
every $a\in\mathcal{A}$,

\item[5.] $\lim_{n\rightarrow\infty}\left\Vert W_{n}\pi_{\mathcal{J}}\left(
a\right)  W_{n}^{\ast}-\rho_{\mathcal{J}}\left(  a\right)  \right\Vert =0$ for
every $a\in\mathcal{A}$.
\end{enumerate}
\end{theorem}

\begin{proof}
First, suppose $x\in\mathcal{A}$ and $x=x^{\ast}$. It follows from \cite{DS}
that there is a sequence $\left\{  U_{n}\right\}  $ such that
\[
\left\Vert U_{n}\pi\left(  x\right)  U_{n}^{\ast}-\rho\left(  x\right)
\right\Vert \rightarrow0.
\]
It follows that $\pi\left(  x\right)  \in\mathcal{J}$ if and only if
$\rho\left(  x\right)  \in\mathcal{J}$ when $x=x^{\ast}$. However, for any
$a\in\mathcal{A}$, we get $\pi\left(  a\right)  \in\mathcal{J}$ if and only if
$\pi\left(  \left\vert a\right\vert \right)  \in\mathcal{J}$. Hence $\pi
^{-1}\left(  \mathcal{J}\right)  =\rho^{-1}\left(  \mathcal{J}\right)  $.
Also, $\pi\left(  a\right)  \in\mathcal{J}_{0}$ if and only if $\mathfrak{R}%
\left(  \pi\left(  a\right)  \right)  \in\mathcal{J}_{0}$. Since
$\mathfrak{R}\left(  \pi\left(  a\right)  \right)  $ and $\mathfrak{R}\left(
\rho\left(  a\right)  \right)  $ are Murray von Neumann equivalent (from
$\left(  2\right)  $), we see that $\pi\left(  a\right)  \in\mathcal{J}_{0}$
if and only if $\rho\left(  a\right)  \in\mathcal{J}_{0}$. It follows that
$\pi^{-1}\left(  \mathcal{J}_{0}\right)  \cap\mathcal{A}_{n}=$ $\rho
^{-1}\left(  \mathcal{J}_{0}\right)  \cap\mathcal{A}_{n}$ for each
$n\in\mathbb{N}$, and, from Lemma \ref{ideal},%
\[
\left[  \bigcup_{n=1}^{\infty}\pi^{-1}\left(  \mathcal{J}_{0}\right)
\cap\mathcal{A}_{n}\right]  ^{-\left\Vert {}\right\Vert }=\left[
\bigcup_{n=1}^{\infty}\rho^{-1}\left(  \mathcal{J}_{0}\right)  \cap
\mathcal{A}_{n}\right]  ^{-\left\Vert {}\right\Vert }=\pi^{-1}\left(
\mathcal{J}\right)  =\rho^{-1}\left(  \mathcal{J}\right)  \text{.}%
\]

Since $\mathcal{A}$ is an ASH algebra, we can assume that there is a sequence%
\[
\mathcal{A}_{1}\subset\mathcal{A}_{2}\subset\cdots
\]
of subalgebras of $\mathcal{A}$ such that $\cup_{n=1}^{\infty}\mathcal{A}_{n}$
is norm dense in $\mathcal{A}$ such that, for each $n\in\mathbb{N}$,%
\[
\mathcal{A}_{n}^{\# \#}=\mathcal{M}_{k\left(  n,1\right)  }\left(  C\left(
X_{n,1}\right)  \oplus\cdots\oplus\mathcal{M}_{k\left(  n,s_{n}\right)
}\left(  X_{n,s_{n}}\right)  \right)
\]
with $X_{n,1},\ldots,X_{n,s_{n}}$ compact Hausdorff spaces.

Suppose $T=\left(  f_{ij}\right)  \in\mathbb{M}_{k}\left(  C\left(  X\right)
\right)  $ is a $k\times k$ matrix of functions. We define $T^{\maltese}=$
\textrm{diag}$\left(  f,f,\ldots,f\right)  $ where $f=\sum_{i,j=1}%
^{k}\left\vert f_{ij}\right\vert ^{2}.$ If $\left\{  e_{ij}:1\leq i,j\leq
n\right\}  $ is the system of matrix units for $\mathbb{M}_{n}\left(
\mathbb{C}\right)  ,$ then $T=\sum_{i,j=1}^{n}f_{ij}e_{ij}$. It is clear that
if $T\geq0,$ then $\mathfrak{R}\left(  T\right)  \leq\mathfrak{R}\left(
T^{\maltese}\right)  $. Since $f_{ij}e_{ss}=e_{si}Te_{js},$ we have
\[
\text{ }\left\vert f_{ij}\right\vert ^{2}e_{ss}=\left(  e_{si}Te_{js}\right)
^{\ast}\left(  e_{si}Te_{js}\right)  =e_{sj}T^{\ast}e_{is}e_{si}Te_{js}%
=e_{js}^{\ast}T^{\ast}e_{ii}Te_{js}.
\]
Thus
\[
T^{\maltese}=\sum_{s=1}^{g}\sum_{i,j=1}^{k}\left\vert f_{ij}\right\vert
^{2}e_{ss}=\sum_{s=1}^{g}\sum_{i,j=1}^{k}e_{js}^{\ast}T^{\ast}e_{ii}Te_{js}.
\]
Suppose $A=A_{1}\oplus\cdots\oplus A_{s_{n}}\in\mathcal{A}_{n}^{\# \#}$. We
define $\Delta_{n}:\mathcal{A}_{n}^{\# \#}\rightarrow\mathcal{Z}\left(
\mathcal{A}_{n}^{\# \#}\right)  $ by%
\[
\Delta_{n}\left(  A\right)  =A_{1}^{\maltese}\oplus\cdots\oplus A_{s_{n}%
}^{\maltese}\text{.}%
\]
Thus if $A\in\mathcal{A}_{n}^{\# \#}$, then $\Delta_{n}\left(  A\right)  $ has
the form
\[
\Delta_{n}\left(  A\right)  =\sum_{k=1}^{m}B_{k}AC_{k},
\]
with $B_{1},C_{1},\ldots,B_{m},C_{m}\in\mathcal{A}_{n}^{\# \#}$.

It is clear that

\begin{enumerate}
\item[a.] $\Delta_{n}\left(  \mathcal{A}_{n}^{\# \#}\right)  $ is contained in
the center $\mathcal{Z}\left(  \mathcal{A}_{n}^{\# \#}\right)  $ of
$\mathcal{A}_{n}^{\# \#}$, and

\item[b.] If $A\geq0,$ then $\mathfrak{R}\left(  A\right)  \leq\mathfrak{R}%
\left(  \Delta_{n}\left(  A\right)  \right)  \in\mathcal{Z}\left(
\mathcal{A}_{n}^{\# \#}\right)  $.
\end{enumerate}

We call a projection $Q\in\mathcal{A}_{n}^{\# \#}$ \textbf{good} if

\begin{enumerate}
\item[c.] $\hat{\pi}\left(  Q\right)  ,\hat{\rho}\left(  Q\right)
\in\mathcal{J}_{0}$

\item[d.] $Q\in\left[  \mathcal{A}_{n}\cap\pi^{-1}\left(  \mathcal{J}%
_{0}\right)  \right]  ^{-\text{\textrm{weak*}}}$

\item[e.] For all $T\in Q\mathcal{A}^{\# \#}Q$, $\mathcal{M}$-rank$\left(
\hat{\pi}\left(  T\right)  \right)  =\mathcal{M}$-rank$\left(  \hat{\rho
}\left(  T\right)  \right)  .$
\end{enumerate}

Our proof is based on four claims.

\textbf{Claim 0:} Suppose $Q_{1},Q_{2}\in\mathcal{A}_{n}^{\# \#}$ are good
projections and $Q_{1}\perp Q_{2}$. Then $Q=Q_{1}+Q_{2}$ is a good projection.
It is clear that $Q$ satisfies $\left(  c\right)  $ and $\left(  d\right)  $.
Let $P=\hat{\pi}\left(  Q\right)  \vee\hat{\rho}\left(  Q\right)
\in\mathcal{J}_{0}$. Thus $P$ is a finite projection in $\mathcal{M}$, so
$P\mathcal{M}P$ is a finite von Neumann algebra. Let $\Phi_{P}:P\mathcal{M}%
P\rightarrow\mathcal{Z}\left(  P\mathcal{M}P\right)  $ be the center-valued
trace. Since $Q_{1}$ and $Q_{2}$ are good, we know from Lemma \ref{trace}
that
\[
\Phi_{P}\circ\hat{\pi}|_{Q_{k}\mathcal{A}^{\# \#}Q_{k}}=\Phi_{P}\circ\hat
{\rho}|_{Q_{k}\mathcal{A}^{\# \#}Q_{k}}%
\]
for $k=1,2.$ Since $Q_{1}\perp Q_{2}$, we know $\hat{\pi}\left(  Q_{1}\right)
\perp\hat{\pi}\left(  Q_{2}\right)  $ and $\hat{\rho}\left(  Q_{1}\right)
\perp\hat{\rho}\left(  Q_{2}\right)  $. Thus if $1\leq i\neq j\leq2$, then,
since $\Phi_{P}$ is tracial, if $A\in\mathcal{A}^{\# \#}$, then,%
\[
\Phi_{P}\left(  \hat{\pi}\left(  Q_{i}AQ_{j}\right)  \right)  =\Phi_{P}\left(
\hat{\pi}\left(  Q_{i}\right)  \hat{\pi}\left(  A\right)  \hat{\pi}\left(
Q_{j}\right)  ^{2}\right)
\]%
\[
=\Phi_{P}\left(  \hat{\pi}\left(  Q_{j}\right)  \hat{\pi}\left(  Q_{i}\right)
\hat{\pi}\left(  A\right)  \hat{\pi}\left(  Q_{j}\right)  \right)  =0.
\]
Similarly,%
\[
\Phi_{P}\left(  \hat{\rho}\left(  Q_{i}AQ_{j}\right)  \right)  =0.
\]
Thus%
\[
\Phi_{P}\left(  \hat{\pi}\left(  QAQ\right)  \right)  =\Phi_{P}\left(
\hat{\pi}\left(  Q_{1}AQ_{1}\right)  \right)  +\Phi_{P}\left(  \hat{\pi
}\left(  Q_{2}AQ_{2}\right)  \right)
\]%
\[
=\Phi_{P}\left(  \hat{\rho}\left(  Q_{1}AQ_{1}\right)  \right)  +\Phi
_{P}\left(  \hat{\rho}\left(  Q_{2}AQ_{2}\right)  \right)  \text{ .}%
\]
Thus, by Lemma \ref{trace}, $Q$ satisfies $\left(  e\right)  $. Hence $Q$ is a
good projection. This proves the claim. A simple induction proof implies that
the sum of a finite pairwise orthogonal family of good projections is good.

\textbf{Claim 1:} If $Q\in\mathcal{A}_{n}^{\# \#}$ is a good projection, then
there is a good projection $P\in\mathcal{Z}\left(  \mathcal{A}_{n}^{\#
\#}\right)  $ such that $Q\leq P.$

\textbf{Proof:} Suppose $Q\in\mathcal{A}_{n}^{\# \#}$ is a good projection.
Choose $B_{1},C_{1},\ldots,B_{k},C_{k}$ in $\mathcal{A}_{n}^{\# \#}$ such that%
\[
E\underset{\text{\textrm{def}}}{=}\sum_{k=1}^{m}B_{k}QC_{k}=\Delta_{n}\left(
Q\right)  \in\mathcal{Z}\left(  \mathcal{A}_{n}^{\# \#}\right)  .
\]
Since $\mathfrak{R}\left(  E\right)  \in\mathcal{Z}\left(  \mathcal{A}_{n}^{\#
\#}\right)  $ and $E\geq0$, we see that%
\[
E=\mathfrak{R}\left(  E\right)  E\mathfrak{R}\left(  E\right)  =\sum_{k=1}%
^{m}\left[  \mathfrak{R}\left(  E\right)  B_{k}\mathfrak{R}\left(  E\right)
\right]  Q\left[  \mathfrak{R}\left(  E\right)  C_{k}\mathfrak{R}\left(
E\right)  \right]  .
\]
Hence we can assume, for $1\leq k\leq m$, that $B_{k},C_{k}\in\mathfrak{R}%
\left(  E\right)  \mathcal{A}^{\# \#}\mathfrak{R}\left(  E\right)  .$

Since $\hat{\pi}\left(  Q\right)  $, $\hat{\rho}\left(  Q\right)
\in\mathcal{J}_{0}$, we see that $\hat{\pi}\left(  E\right)  $ and $\hat{\rho
}\left(  E\right)  \in\mathcal{J}_{0}$, which, in turn, implies $\hat{\pi
}\left(  \mathfrak{R}\left(  E\right)  \right)  $ and $\hat{\rho}\left(
\mathfrak{R}\left(  E\right)  \right)  \in\mathcal{J}_{0}$. Then $F=\hat{\pi
}\left(  \mathfrak{R}\left(  E\right)  \right)  \vee\hat{\rho}\left(
\mathfrak{R}\left(  E\right)  \right)  \in\mathcal{J}_{0}$ is a finite
projection. Thus $F\mathcal{M}F$ is a finite von Neumann algebra. Also, since,
for $1\leq k\leq m$, $B_{k},C_{k}\in$ $\mathfrak{R}\left(  E\right)
\mathcal{A}_{n}^{\# \#}\mathfrak{R}\left(  E\right)  ,$ we see that $\hat{\pi
}\left(  B_{k}QC_{k}\right)  ,\hat{\rho}\left(  B_{k}QC_{k}\right)  \in
F\mathcal{M}F$. Let $\Phi_{F}$ be the center-valued trace on $F\mathcal{M}F$.
Since $Q$ is a good projection and in $E\mathcal{A}^{\# \#}E$, we know from
Lemma \ref{trace}, that for every $A\in\mathcal{A}^{\# \#}$,%
\[
\Phi_{F}\left(  \hat{\pi}\left(  QAQ\right)  \right)  =\Phi_{F}\left(
\hat{\rho}\left(  QAQ\right)  \right)  .
\]
Now $\hat{\pi},\hat{\rho}:E\mathcal{A}^{\# \#}E\rightarrow F\mathcal{M}F$ are
$\ast$-homomorphisms, and, since $\Phi_{F}$ is tracial, we see for
$A\in\mathcal{A}^{\# \#}$,
\[
\Phi_{F}\left(  \hat{\pi}\left(  EAE\right)  \right)  =
\]%
\[
=\sum_{j,k=1}^{m}\Phi_{F}\left(  \left[  \hat{\pi}\left(  B_{k}\right)
\hat{\pi}\left(  Q\right)  \right]  \left[  \hat{\pi}\left(  Q\right)
\hat{\pi}\left(  C_{k}\right)  \hat{\pi}\left(  A\right)  \hat{\pi}\left(
B_{j}\right)  \hat{\pi}\left(  Q\right)  \hat{\pi}\left(  C_{j}\right)
\right]  \right)
\]%
\[
=\sum_{j,k=1}^{m}\Phi_{F}\left(  \left[  \hat{\pi}\left(  Q\right)  \hat{\pi
}\left(  C_{k}\right)  \hat{\pi}\left(  A\right)  \hat{\pi}\left(
B_{j}\right)  \hat{\pi}\left(  Q\right)  \hat{\pi}\left(  C_{j}\right)
\right]  \left[  \hat{\pi}\left(  B_{k}\right)  \hat{\pi}\left(  Q\right)
\right]  \right)
\]%
\[
=\sum_{j,k=1}^{m}\Phi_{F}\left(  \hat{\pi}\left(  QC_{k}AB_{j}QC_{j}%
B_{k}Q\right)  \right)  =\sum_{j,k=1}^{m}\Phi_{F}\left(  \hat{\rho}\left(
QC_{k}AB_{j}QC_{j}B_{k}Q\right)  \right)
\]%
\[
=\sum_{j,k=1}^{m}\Phi_{F}\left(  \left[  \hat{\rho}\left(  Q\right)  \hat
{\rho}\left(  C_{k}\right)  \hat{\rho}\left(  A\right)  \hat{\rho}\left(
B_{j}\right)  \hat{\rho}\left(  Q\right)  \hat{\rho}\left(  C_{j}\right)
\right]  \left[  \hat{\rho}\left(  B_{k}\right)  \hat{\rho}\left(  Q\right)
\right]  \right)
\]%
\[
=\Phi_{F}\left(  \hat{\rho}\left(  EAE\right)  \right)  \text{.}%
\]
Thus $\Phi_{F}\circ\hat{\pi}=\Phi_{F}\circ\hat{\rho}$ on $E\mathcal{A}^{\#
\#}E$, and since $\hat{\pi},\hat{\rho},$ and $\Phi_{F}$ are weak* continuous,
we have $\Phi_{F}\circ\hat{\pi}=\Phi_{F}\circ\hat{\rho}$ on $\left(
E\mathcal{A}^{\# \#}E\right)  ^{-\text{\textrm{weak}*}}=\mathfrak{R}\left(
E\right)  \mathcal{A}^{\# \#}\mathfrak{R}\left(  E\right)  $.

Finally, since $\left[  \mathcal{A}_{n}\cap\pi^{-1}\left(  \mathcal{J}%
_{0}\right)  \right]  ^{-\text{\textrm{weak*}}}$ is a weak* closed $\ast
$-algebra, and an ideal for $\mathcal{A}_{n}^{\# \#}$, we see that
\[
E=\Delta_{n}\left(  Q\right)  =\sum_{k=1}^{m}B_{k}QC_{k}\in\left[
\mathcal{A}_{n}\cap\pi^{-1}\left(  \mathcal{J}_{0}\right)  \right]
^{-\text{\textrm{weak*}}},
\]
so $P=\mathfrak{R}\left(  E\right)  \in\left[  \mathcal{A}_{n}\cap\pi
^{-1}\left(  \mathcal{J}_{0}\right)  \right]  ^{-\text{\textrm{weak*}}}$. Thus
$P=\mathfrak{R}\left(  E\right)  \in\mathcal{Z}\left(  \mathcal{A}_{n}^{\#
\#}\right)  $ is a good projection and $Q\leq P$.

\textbf{Claim 2: }If $Q_{1},Q_{2}\in\mathcal{A}_{n}^{\# \#}$ are good
projections, then there is a good projection $Q\in\mathcal{Z}\left(
\mathcal{A}_{n}^{\# \#}\right)  $ such that $Q_{1},Q_{2}\leq Q$.

\textbf{Proof:} By Claim 1 we can choose good projections $P_{1},P_{2}%
\in\mathcal{Z}\left(  \mathcal{A}_{n}^{\# \#}\right)  $ such that $Q_{1}\leq
P_{1}$ and $Q_{2}\leq P_{2}$. Since $P_{1}$ and $P_{2}$ commute and
$P_{1}\left(  1-P_{2}\right)  \leq P_{1}$, $P_{1}P_{2}\leq P_{1}$ and $\left(
1-P_{1}\right)  P_{2}\leq P_{2}$, we see that $\left\{  P_{1}\left(
1-P_{2}\right)  ,P_{1}P_{2},\left(  1-P_{1}\right)  P_{2}\right\}  $ is an
orthogonal family of good projections. Thus, by Case 0,
\[
Q=P_{1}\vee P_{2}=P_{1}\left(  1-P_{2}\right)  +P_{1}P_{2}+\left(
1-P_{1}\right)  P_{2}%
\]
is a good projection in $\mathcal{Z}\left(  \mathcal{A}^{\# \#}\right)  $.
Thus Claim 2 is proved.

\textbf{Claim 3:} If $0\leq x\in\mathcal{A}_{n}\cap\pi^{-1}\left(
\mathcal{J}_{0}\right)  $, then $\mathfrak{R}\left(  \Delta_{n}\left(
x\right)  \right)  \in\mathcal{Z}\left(  \mathcal{A}_{n}^{\# \#}\right)  $ is good.

\textbf{Proof}: We know that $\hat{\pi}\left(  \mathfrak{R}\left(  x\right)
\right)  $ and $\hat{\rho}\left(  \mathfrak{R}\left(  x\right)  \right)  $ are
Murray von Neumann equivalent and $\mathcal{M}$-rank$\left(  \pi\left(
x\right)  \right)  $ and $\mathcal{M}$-rank$\left(  \rho\left(  x\right)
\right)  $ are equal. Since $\pi\left(  x\right)  \in\mathcal{J}_{0}$, we know
$\hat{\pi}\left(  \mathfrak{R}\left(  x\right)  \right)  $,$\hat{\rho}\left(
\mathfrak{R}\left(  x\right)  \right)  \in\mathcal{J}_{0}$. Arguing as in the
proof of Claim 1, we see that $F=\hat{\pi}\left(  \mathfrak{R}\left(
x\right)  \right)  \vee\hat{\rho}\left(  \mathfrak{R}\left(  x\right)
\right)  \in\mathcal{J}_{0}$ and that%
\[
\hat{\pi},\hat{\rho}:\left[  x\mathcal{A}x\right]  ^{-\left\Vert {}\right\Vert
}\rightarrow F\mathcal{M}F
\]
satisfy $\Phi_{F}\circ\hat{\pi}=\Phi_{F}\circ\hat{\rho}$. Thus $\Phi_{F}%
\circ\hat{\pi}=\Phi_{F}\circ\hat{\rho}$ on $\left[  x\mathcal{A}x\right]
^{-\text{\textrm{weak*}}}=\mathfrak{R}\left(  x\right)  \mathcal{A}^{\#
\#}\mathfrak{R}\left(  x\right)  $. Thus $\mathfrak{R}\left(  x\right)  $ is a
good projection. This proves Claim 3.

We can choose a countable dense set $\left\{  b_{1},b_{2},\ldots\right\}  $ of
$\cup_{n=1}^{\infty}\left(  \mathcal{A}_{n}\cap\pi^{-1}\left(  \mathcal{J}%
_{0}\right)  \right)  $ whose closure is $\pi^{-1}\left(  \mathcal{J}\right)
$.

We now want to define a sequence $0=P_{0}\leq P_{1}\leq P_{2}\leq\cdots$ of
good projections such that

\begin{enumerate}
\item $P_{n}\in\mathcal{Z}\left(  \mathcal{A}_{n}^{\# \#}\right)  $ for all
$n\in\mathbb{N}$,

\item If $1\leq k\leq n$ and $b_{k}\in\mathcal{A}_{n}$, then $\mathfrak{R}%
\left(  b_{k}\right)  \leq P_{n}$, i.e.,
\[
b_{k}=P_{n}b_{k}%
\]

\end{enumerate}

Define $P_{0}=0$. Suppose $n\in\mathbb{N}$ and $P_{k}$ has been defined for
$0\leq k\leq n$. We let $x_{n}=\sum_{k\leq n+1,b_{k}\in\mathcal{A}_{n+1}}%
b_{k}b_{k}^{\ast}\in\mathcal{A}_{n+1}\cap\pi^{-1}\left(  \mathcal{J}%
_{0}\right)  $. Thus, by Claim 3, $P_{n}$ and $\mathfrak{R}\left(
\Delta_{n+1}\left(  x_{n}\right)  \right)  $ are good projections in
$\mathcal{A}_{n}^{\# \#}$, and they commute since $\mathfrak{R}\left(
\Delta_{n+1}\left(  x_{n}\right)  \right)  \in\mathcal{Z}\left(
\mathcal{A}_{n+1}^{\# \#}\right)  $. By Claim 2, there is a good projection
$P_{n+1}\in\mathcal{Z}\left(  \mathcal{A}_{n+1}^{\# \#}\right)  $ such that
$P_{n}\leq P_{n+1}$ and $\mathfrak{R}\left(  \Delta_{n+1}\left(  x_{n}\right)
\right)  \leq P_{n+1}$. Clearly, if $1\leq k\leq n$ and $b_{k}\in
\mathcal{A}_{n}$, we have $\mathfrak{R}\left(  b_{k}\right)  =\mathfrak{R}%
\left(  b_{k}b_{k}^{\ast}\right)  \leq\mathfrak{R}\left(  x_{n}\right)  \leq
P_{n+1}$.

Since $P_{n}$ is a good projection, $P_{n}\in\left[  \mathcal{A}_{n}\cap
\pi^{-1}\left(  \mathcal{J}_{0}\right)  \right]  ^{-\text{\textrm{weak*}}}.$
Thus%
\[
P_{n}\leq\sup\left\{  \mathfrak{R}\left(  x\right)  :x\in\mathcal{A}_{n}%
\cap\pi^{-1}\left(  \mathcal{J}_{0}\right)  \right\}  \in\mathcal{A}_{n}^{\#
\#}.
\]
Thus $\hat{\pi}\left(  P_{n}\right)  \leq P_{\pi,\mathcal{J}}$ (the projection
onto $H_{\pi,\mathcal{J}}$) and $\hat{\rho}\left(  P_{n}\right)  \leq
P_{\rho,\mathcal{J}}$ (the projection onto $H_{\rho,\mathcal{J}}$). Let
$P_{e}=\lim_{n\rightarrow\infty}P_{n}$ (weak*). Thus $\hat{\pi}\left(
P_{e}\right)  \leq P_{\pi,\mathcal{J}}$ and $\hat{\rho}\left(  P_{e}\right)
\leq P_{\rho,\mathcal{J}}$. On the other hand, for every $k\in\mathbb{N}$,%
\[
\lim_{n\rightarrow\infty}\left\Vert b_{k}-P_{n}b_{k}\right\Vert =0.
\]
This implies
\[
P_{e}b=b\text{ for every }b\in\left[  \pi^{-1}\left(  \mathcal{J}\right)
\right]  ^{-\left\Vert {}\right\Vert }\text{ .}%
\]

Thus $\hat{\pi}\left(  P_{e}\right)  =P_{\pi,\mathcal{J}}$ and $\hat{\rho
}\left(  P_{e}\right)  =P_{\rho,\mathcal{J}}$. Thus $P_{\pi,\mathcal{J}}$ and
$P_{\rho,\mathcal{J}}$ are Murray von Neumann equivalent.

Since $P_{n}\in\mathcal{A}_{n}^{\prime}$ for each $n\in\mathbb{N}$, we have of
every $A\in\cup_{k=1}^{\infty}\mathcal{A}_{k}$,%
\[
\lim_{n\rightarrow\infty}\left\Vert AP_{n}-P_{n}A\right\Vert =0.
\]
Hence,
\[
\lim_{n\rightarrow\infty}\left\Vert AP_{n}-P_{n}A\right\Vert =0
\]
holds for every $A\in\mathcal{A}$.

Choose a dense subset $\left\{  A_{1},A_{2},\ldots\right\}  $ of $\mathcal{A}%
$. Suppose and $m\in\mathbb{N}$. It follows that we can choose a subsequence
$\left\{  P_{n_{k}}\right\}  $ of $\left\{  P_{n}\right\}  $ such that, for
all $1\leq n<\infty,$%
\[
\sum_{k=1}^{\infty}\left\Vert A_{n}P_{n_{k}}-P_{n_{k}}A_{n}\right\Vert
<\infty\text{,}%
\]
and, for $1\leq n\leq m$,%
\[
\sum_{k=1}^{\infty}\left\Vert A_{n}P_{n_{k}}-P_{n_{k}}A_{n}\right\Vert
<\frac{1}{8m}\text{ .}%
\]
Define $e_{k}=P_{n_{k}}-P_{n_{k-1}}$ (with $P_{n_{0}}=0$) and define
$\varphi:\mathcal{A}\rightarrow\sum_{1\leq k<\infty}^{\oplus}e_{k}%
\mathcal{A}e_{k}$ by%
\[
\varphi\left(  T\right)  =\sum_{k=1}^{\infty}e_{k}Te_{k}\text{ .}%
\]
It follows from \cite[page 903]{PRH} that the above conditions on $\left\Vert
A_{n}P_{n_{k}}-P_{n_{k}}A_{n}\right\Vert $ that, for all $k\in\mathbb{N}$,%
\[
A_{k}-\varphi\left(  A_{k}\right)  \in\hat{\pi}^{-1}\left(  \mathcal{J}%
\right)  \cap\hat{\rho}^{-1}\left(  \mathcal{J}\right)  \text{ }%
\]
and
\[
\left\Vert P_{e}A_{n}-\varphi\left(  A_{n}\right)  \right\Vert <\frac{1}%
{4m}\text{ .}%
\]
for $1\leq n\leq m$.

Suppose $k\in\mathbb{N}$. For each $n\geq n_{k}$, $e_{k}\mathcal{A}_{n}%
e_{k}\subset\mathcal{A}_{n}^{\# \#}$, which is homogeneous. Hence $C^{\ast
}\left(  e_{k}\mathcal{A}_{n}e_{k}\right)  $ is subhomogeneous. Thus $C^{\ast
}\left(  e_{k}\mathcal{A}e_{k}\right)  $ is ASH. If we let

$E_{k}=\hat{\pi}\left(  e_{k}\right)  \vee\hat{\rho}\left(  e_{k}\right)  $
for each $k\in\mathbb{N}$, we have $E_{k}$ is a finite projection,
$E_{k}\mathcal{M}E_{k}$ is a finite von Neumann algebra,%
\[
\hat{\pi},\hat{\rho}:C^{\ast}\left(  e_{k}\mathcal{A}e_{k}\right)  \rightarrow
E_{k}\mathcal{M}E_{k},
\]
and, if Let $\Phi_{E_{k}}$ is the center-valued trace on $E_{k}\mathcal{M}%
E_{k}$, then
\[
\Phi_{E_{k}}\circ\left(  \hat{\pi}|_{C^{\ast}\left(  e_{k}\mathcal{A}%
e_{k}\right)  }\right)  =\Phi_{E_{k}}\circ\left(  \hat{\rho}|_{C^{\ast}\left(
e_{k}\mathcal{A}e_{k}\right)  }\right)  ,
\]
and $C^{\ast}\left(  e_{k}\mathcal{A}e_{k}\right)  $ is ASH, it follows from
Theorem \ref{SAFE} that%
\[
\hat{\pi}|_{C^{\ast}\left(  e_{k}\mathcal{A}e_{k}\right)  }\sim_{a}\hat{\rho
}|_{C^{\ast}\left(  e_{k}\mathcal{A}e_{k}\right)  \text{ }}\text{ }\left(
E_{k}\mathcal{M}E_{k}\right)  .
\]
Since $\hat{\pi}\left(  e_{k}\right)  $ and $\hat{\rho}\left(  e_{k}\right)  $
are projections, then by \cite[Proposition 5.2.6]{W}, any unitary that
conjugates $\hat{\pi}\left(  e_{k}\right)  $ to a projection that is really
close to $\hat{\rho}\left(  e_{k}\right)  $ is close to a unitary that
conjugates $\hat{\pi}\left(  e_{k}\right)  $ exactly to $\hat{\rho}\left(
e_{k}\right)  $. We can therefore, for each $k\in\mathbb{N}$, choose a unitary
$U_{k}\in E_{k}\mathcal{M}E_{k}$ such that
\[
\left\Vert U_{k}\hat{\pi}\left(  e_{k}a_{n}e_{k}\right)  U_{k}^{\ast}%
-\hat{\rho}\left(  e_{k}a_{n}e_{k}\right)  \right\Vert <\frac{1}{4km}%
\]
when $1\leq n\leq k+m<\infty$, and such that
\[
U_{k}\hat{\pi}\left(  e_{k}\right)  U_{k}^{\ast}=\rho\left(  e_{k}\right)  .
\]

For each $k\in\mathbb{N}$, let $V_{k}=U_{k}\hat{\pi}\left(  e_{k}\right)  $.
Then $V_{k}$ is a partial isometry whose initial projection is $\hat{\pi
}\left(  e_{k}\right)  =V_{k}^{\ast}V_{k}$ and final projection is $\hat{\rho
}\left(  e_{k}\right)  =V_{k}V_{k}^{\ast}$. Also%
\[
\left\Vert V_{k}\hat{\pi}\left(  e_{k}\right)  \pi\left(  a_{n}\right)
\hat{\pi}\left(  e_{k}\right)  V_{k}^{\ast}-\hat{\rho}\left(  e_{k}\right)
\rho\left(  a_{n}\right)  \hat{\rho}\left(  e_{k}\right)  \right\Vert
<\frac{1}{4km}%
\]
for $1\leq n\leq k+m<\infty$. Then $W_{m}=\sum_{k=1}^{\infty}V_{k}$ is a
partial isometry in $\mathcal{M}$ with initial projection $\hat{\pi}\left(
P_{e}\right)  =P_{\pi,\mathcal{J}}$ and final projection $\hat{\rho}\left(
P_{e}\right)  =P_{\rho,\mathcal{J}}$. Moreover,
\[
W_{m}\hat{\pi}\left(  \varphi\left(  a_{n}\right)  \right)  W_{m}^{\ast}%
=\sum_{1\leq k<\infty}^{\oplus}V_{k}\hat{\pi}\left(  e_{k}a_{n}e_{k}\right)
V_{k}^{\ast},
\]
and
\[
\hat{\rho}\left(  \varphi\left(  a_{n}\right)  \right)  =\sum_{1\leq k<\infty
}^{\oplus}\hat{\rho}\left(  e_{k}a_{n}e_{k}\right)  .
\]
Since $V_{k}\hat{\pi}\left(  e_{k}a_{n}e_{k}\right)  V_{k}^{\ast}$, $\hat
{\rho}\left(  e_{k}a_{n}e_{k}\right)  \in\mathcal{J}$ for each $n,k\in
\mathbb{N}$ and since%
\[
\lim_{k\rightarrow\infty}\left\Vert V_{k}\hat{\pi}\left(  e_{k}a_{n}%
e_{k}\right)  V_{k}^{\ast}-\hat{\rho}\left(  e_{k}a_{n}e_{k}\right)
\right\Vert =0,
\]
we see that
\[
W_{m}\hat{\pi}\left(  \varphi\left(  a_{n}\right)  \right)  W_{m}^{\ast}%
-\hat{\rho}\left(  \varphi\left(  a_{n}\right)  \right)  \in\mathcal{J}%
\]
for every $n\in\mathbb{N}$. Also,
\[
\left\Vert W_{m}\hat{\pi}\left(  \varphi\left(  a_{n}\right)  \right)
W_{m}^{\ast}-\hat{\rho}\left(  \varphi\left(  a_{n}\right)  \right)
\right\Vert <\frac{1}{4m}%
\]
for $1\leq n\leq m$.

Also
\[
\hat{\pi}\left(  \varphi\left(  a_{n}\right)  \right)  -\pi\left(
a_{n}\right)  =\hat{\pi}\left(  \varphi\left(  a_{n}\right)  -a_{n}\right)
\in\mathcal{J}%
\]
and
\[
\hat{\pi}\left(  \varphi\left(  a_{n}\right)  \right)  -\rho\left(
a_{n}\right)  =\hat{\rho}\left(  \varphi\left(  a_{n}\right)  -a_{n}\right)
\in\mathcal{J}%
\]
for every $n\in\mathbb{N}$ and
\[
\left\Vert \hat{\pi}\left(  \varphi\left(  a_{n}\right)  \right)  -\pi\left(
a_{n}\right)  \right\Vert <\frac{1}{4m}\text{ and }\left\Vert \hat{\rho
}\left(  \varphi\left(  a_{n}\right)  \right)  -\rho\left(  a_{n}\right)
\right\Vert <\frac{1}{4m}%
\]
for $1\leq n\leq m$.

For each $n\in\mathbb{N}$,%
\[
W_{m}\pi\left(  a_{n}\right)  W_{m}^{\ast}-\rho\left(  a_{n}\right)
\]%
\[
=\left[  W_{m}\left(  \pi\left(  a_{n}\right)  -\hat{\pi}\left(
\varphi\left(  a_{n}\right)  \right)  \right)  W_{m}^{\ast}\right]  +\left[
W_{m}\hat{\pi}\left(  \varphi\left(  a_{n}\right)  \right)  W_{m}^{\ast}%
-\hat{\rho}\left(  \varphi\left(  a_{n}\right)  \right)  \right]
\]%
\[
+\hat{\rho}\left(  \varphi\left(  a_{n}\right)  \right)  -\rho\left(
a_{n}\right)  .
\]
Thus, for every $n\in\mathbb{N}$,%
\[
W_{m}\pi\left(  a_{n}\right)  W_{m}^{\ast}-\rho\left(  a_{n}\right)
\in\mathcal{J}.
\]
Also, for $1\leq n\leq m$,%
\[
\left\Vert W_{m}\pi\left(  a_{n}\right)  W_{m}^{\ast}-\rho\left(
a_{n}\right)  \right\Vert <\frac{1}{m}\text{ .}%
\]
It follows, for every $a\in\mathcal{A}$, that
\[
W_{m}\hat{\pi}\left(  \varphi\left(  a\right)  \right)  W_{m}^{\ast}-\hat
{\rho}\left(  \varphi\left(  a\right)  \right)  \in\mathcal{J}%
\]
and%
\[
\lim_{m\rightarrow\infty}\left\Vert W_{m}\pi\left(  a\right)  W_{m}^{\ast
}-\rho\left(  a\right)  \right\Vert =0\text{.}%
\]

\end{proof}

\begin{remark}
In two cases, namely, when $H_{\pi,\mathcal{J}}=H_{\rho,\mathcal{J}}=H,$ or
when $\pi\left(  \cdot\right)  |_{H_{\pi,\mathcal{J}}^{\perp}}$ and
$\rho\left(  \cdot\right)  |_{H_{\rho,\mathcal{J}}^{\perp}}$ are unitarily
equivalent, the conclusion in Theorem \ref{main} becomes%
\[
\pi\sim_{a}\rho\text{ }\left(  \mathcal{J}\right)  \text{ .}%
\]
\bigskip
\end{remark}

When $\mathcal{A}$ is a separable ASH C*-algebra and $\mathcal{M}$ is a
sigma-finite $II_{\infty}$ factor von Neumann algebra, we can use Theorems
\ref{main} and \ref{nuclear} to have both parts of Voiculescu's theorem,
including an extension of results in \cite{DH}. \bigskip

\begin{corollary}
Suppose $\mathcal{A}$ is a separable ASH C*-algebra, $\mathcal{M}$ is a
sigma-finite type $II_{\infty}$ factor von Neumann algebra on a Hilbert space
$H$, and $\tau$ is a faithful normal tracial weight on $\mathcal{M}$. Suppose
$\pi,\rho:\mathcal{A}\rightarrow\mathcal{M}$ are unital $\ast$-homomorphisms
such that, for every $a\in\mathcal{A}$%
\[
\mathcal{M}\text{-\textrm{rank}}\left(  \pi\left(  a\right)  \right)
=\mathcal{M}\text{-\textrm{rank}}\left(  \rho\left(  a\right)  \right)
\text{.}%
\]
Then $\pi\sim_{a}\rho$ $\left(  \mathcal{K}_{\mathcal{M}}\right)  $.
\end{corollary}

\begin{theorem}
\label{ASHmain}Suppose $\mathcal{M}\subset B\left(  H\right)  $ is a
semifinite von Neumann algebra with no finite summands, $H$ is separable, and
$\mathcal{A}$ is a separable unital ASH C*-algebra. Also suppose $\pi
,\rho:\mathcal{A}\rightarrow\mathcal{M}$ are unital $\ast$-homomorphisms such
that, for every $a\in\mathcal{A}$%
\[
\mathcal{M}\text{-\textrm{rank}}\left(  \pi\left(  a\right)  \right)
=\mathcal{M}\text{-\textrm{rank}}\left(  \rho\left(  a\right)  \right)  \text{
.}%
\]
Then $\pi\sim_{a}\rho$ $\left(  \mathcal{K}_{\mathcal{M}}\right)  $.
\end{theorem}

\begin{proof}
We can write $\pi=\pi_{\mathcal{K}_{\mathcal{M}}}\oplus\pi_{1}$ and $\rho
=\rho_{\mathcal{K}_{\mathcal{M}}}\oplus\rho_{1}$. It follows from Theorem
\ref{nuclear} that
\[
\pi\sim_{a}\pi_{\mathcal{K}_{\mathcal{M}}}\oplus\pi_{1}\oplus\rho_{1}\text{
}\left(  \mathcal{K}_{\mathcal{M}}\right)  \text{ and }\rho\sim_{a}%
\rho_{\mathcal{K}_{\mathcal{M}}}\oplus\pi_{1}\oplus\rho_{1}\text{ }\left(
\mathcal{K}_{\mathcal{M}}\right)  \text{.}%
\]
It follows from Theorem \ref{main} that $\pi\sim_{a}\rho$ $\left(
\mathcal{K}_{\mathcal{M}}\right)  $.
\end{proof}

\end{document}